\documentclass[12pt,a4paper]{amsart}

\usepackage{amsmath,amssymb,amsfonts}
\usepackage[utf8]{inputenc}
\usepackage[T1]{fontenc}
\usepackage{newtxtext,newtxmath,mathtools}
\usepackage[colorlinks=true,linkcolor=teal,citecolor=magenta,urlcolor=blue]{hyperref}
\usepackage[margin=20mm,top=28mm]{geometry}
\usepackage{enumerate}
\usepackage{tikz}

\newtheorem{theorem}{Theorem}[section]
\newtheorem{corollary}[theorem]{Corollary}
\newtheorem{lemma}[theorem]{Lemma}
\newtheorem{proposition}[theorem]{Proposition}
\theoremstyle{definition}
\newtheorem{example}[theorem]{Example}

\newtheorem{remark}[theorem]{Remark}

\newcommand{\K}{\Bbbk}
\newcommand{\Hilb}{\operatorname{Hilb}}

\newcommand{\Hom}{\operatorname{Hom}}
\newcommand{\GKdim}{\operatorname{GKdim}}
\newcommand{\ord}{\operatorname{ord}}
\newcommand{\coker}{\operatorname{coker}}
\newcommand{\length}{\operatorname{length}}

\renewcommand{\leq}{\leqslant}
\renewcommand{\geq}{\geqslant}
\numberwithin{equation}{section}
\title{Local Smith Profiles of Twisted Calabi--Yau Algebras}

\author{Atabey Kaygun}
\address{Istanbul Technical University, Istanbul, Turkey}
\email{kaygun@itu.edu.tr}

\begin{document}

\begin{abstract}
  The matrix Hilbert series of a locally finite elementary twisted Calabi--Yau algebra is
  the inverse of a matrix polynomial.  The Smith normal form of this polynomial over the
  power series ring at $x=1$ produces a finite list of local exponents refining the
  Gelfand--Kirillov dimension, which records only the largest of them.  We show that the
  Calabi--Yau symmetry makes this local data rigid: the algebra decomposes into ring factors
  according to the average Artin--Schelter index along Nakayama cycles, and after a local
  normalization the symmetry induces a nonsingular linking form on the Smith cokernel
  together with a finite-order Nakayama action on its layers, forcing reciprocal-eigenvalue
  and parity constraints on the multiplicities of the exponents.  We compute the complete
  local data for cyclic skew-group algebras, derive a parity sieve for quiver
  classifications in dimension three, and realize, as an iterated smash product of a graded
  down-up algebra, a four-vertex type that a recent classification had left open.
\end{abstract}

\maketitle

\section{Introduction}

Twisted Calabi--Yau algebras are noncommutative analogues of coordinate rings of varieties
with trivial canonical bundle: the dualizing bimodule is invertible, and the failure of its
triviality is measured by the Nakayama automorphism
\cite{Ginzburg2006,ReyesRogalskiZhang2014}.  When the algebra is connected graded, these are
precisely the Artin--Schelter regular algebras \cite{ArtinSchelter1987,ReyesRogalski2022}.
Allowing instead a semisimple degree-zero part $A_0=\bigoplus_{i=1}^r\K e_i$, spanned by
primitive orthogonal idempotents $e_1,\ldots,e_r$, gives the \emph{elementary} twisted
Calabi--Yau algebras \cite{ReyesRogalski2019}: the quiver generalizations that arise in
noncommutative projective geometry \cite{GaddisRogalski2021}, in the representation theory
of quivers with potential \cite{Ginzburg2006,Bocklandt2008}, and in the skew-group algebras
of the McKay correspondence \cite{ReyesRogalskiZhang2014}.

Let $A$ be such an algebra, locally finite and positively graded, of dimension $d$.  Its
Calabi--Yau symmetry leaves a concrete trace on the matrix Hilbert series
\[
  C_A(x)=\bigl(\Hilb_x(e_iAe_j)\bigr)_{i,j=1}^r.
\]
Reyes--Rogalski \cite{ReyesRogalski2019}
show that $C_A(x)$ is the inverse of a polynomial homological denominator
$q_A(x)$ and that $q_A$ satisfies reciprocal and commutation identities
involving the Nakayama permutation and the vector AS-index.  We ask
what these Calabi--Yau identities force on the local Smith structure of
$q_A$ at $x=1$.  The determinant records only the total vanishing order there;
the individual elementary divisors retain finer information.  In particular, when the
growth is finite, the Gelfand--Kirillov dimension of $A$ is nothing but the largest local
exponent (Lemma~\ref{lem:cy-smith-growth}); the full profile, together with the Nakayama
action constructed below, is a strictly finer invariant obtained at no extra cost.

Put $t=x-1$, $R=\K\llbracket t\rrbracket$, and $M_A=q_A(1+t)$.  Ordinary Smith theory over
the DVR $R$ gives nonnegative exponents $a_1\leq\cdots\leq a_r$ and a finite torsion module
\[
  \mathcal G_A=\coker(M_A:R^r\to R^r)
  \cong\bigoplus_i R/(t^{a_i}).
\]
The Calabi--Yau functional identities make this local module more rigid than an arbitrary
Smith cokernel: they control how the AS-index enters locally and induce a finite-order
Nakayama symmetry on its exact elementary-divisor layers.

The first consequence is global.  If two Nakayama cycles have different average AS-indices,
then they belong to different direct-product factors
(Theorem~\ref{thm:cycle-average-decomposition}); hence every cycle of a ring-indecomposable
algebra has the same average.  Locally, the remaining variation in the AS-index can be
removed by a unitary diagonal change of basis
(Theorem~\ref{thm:local-as-index-normalization}).  After normalization the multiplier is
determined only by the Nakayama permutation and the common cycle average.

This normalized functional equation produces a nonsingular linking pairing on $\mathcal G_A$
(Theorem~\ref{thm:twisted-cy-linking}).  Passing to the exact Smith layers
\[
  V_a(A)=
  \frac{\mathcal G_A[t^a]}
       {\mathcal G_A[t^{a-1}]+t\mathcal G_A[t^{a+1}]}
\]
retains both the multiplicity of the elementary divisor $t^a$ and the induced finite-order
Nakayama action.  The auxiliary forms on these layers pair reciprocal eigenspaces and give
the congruence
\[
  m_{a,-}(A)\equiv(d+a)m_a(A)\pmod2.
\]
For an odd-order Nakayama permutation, every exponent whose parity is opposite to the
Calabi--Yau dimension, including the zero exponent, therefore occurs with even multiplicity.
These restrictions are finer than determinant parity or maximal pole order.

The final section makes the restrictions concrete.  Cyclic skew-group algebras admit a
Fourier diagonalization that computes the complete equivariant local profile.  In dimension
$3$ this yields a short parity sieve for the Gaddis--Rogalski classification
\cite{GaddisRogalski2021} (Subsection~\ref{ex:profile-obstruction}) and repackages one of
their growth exclusions in local Smith terms.  In four vertices the first starred
two-two-cycle candidate of Gaddis--Lamkin--Nguyen--Wright
\cite{GaddisLamkinNguyenWright2024} survives the local restrictions;
Subsection~\ref{ex:four-vertex-realization} constructs it explicitly as an iterated
$C_2$-smash product of the graded down-up algebra $D(0,-1)$
(Theorem~\ref{thm:starred-realization}).  The construction is independent of the obstruction
theory, while the resulting algebra realizes exactly the predicted local profile and
Nakayama-layer actions.

\subsection{Context and prior work}

The vector AS-index, polynomial homological denominator, and reciprocal identities used here
are due to Reyes--Rogalski \cite{ReyesRogalski2019}.  They also show that in the degree-one
indecomposable case the AS-index is scalar, whereas weighted examples show that scalarity
fails in general.  The cycle-average decomposition isolates the rigidity that remains
without the degree-one assumption; its relation to the average Gorenstein parameter of
Iyama--Kimura--Ueyama \cite{IyamaKimuraUeyama2026} is discussed in the remark following
Theorem~\ref{thm:cycle-average-decomposition}.

Local Smith forms and partial multiplicities for matrix polynomials are classical; see for
example Wilkening--Yu \cite{WilkeningYu2011}.  Palindromic and skew-symmetric structures
impose parity restrictions on those multiplicities \cite{MackeyMackeyMehlMehrmann2011}
\cite{MackeyMackeyMehlMehrmann2013}, and M\"obius transformations transport local Smith data
\cite{MackeyMackeyMehlMehrmann2015}.  The auxiliary forms below are the standard DVR
construction associated to linking pairings; compare Levine \cite{Levine1980} and Orson
\cite{Orson2015}.  What is new here is the interaction of these forms with the normalized
finite-order Nakayama action.  The classification applications use Gaddis--Rogalski
\cite{GaddisRogalski2021} and Gaddis--Lamkin--Nguyen--Wright
\cite{GaddisLamkinNguyenWright2024}; the four-vertex realization uses Le Meur's
smash-product theorem \cite{LeMeur2019}.

\subsection{Notation and conventions}

Throughout, $\K$ is a field of characteristic zero, $t=x-1$, $R=\K\llbracket t\rrbracket$,
and $F=\K((t))$.  The involution induced by $x\mapsto x^{-1}$ is denoted by $\iota$, so that
$\iota(t)=-t/(1+t)$; for a matrix $B$ over $F$ we write $B^\dagger=\iota(B)^T$.  For an $R$-module $T$, we will use
$T[t^a]$ to denote $\ker(t^a:T\to T)$.  We use $\ord_t(0)=+\infty$ and
$M\langle m\rangle_n=M_{n-m}$ for grading shifts.  For $c\in\mathbb Q$, $(1+t)^c$ denotes
its binomial series in $R$.

\section{Calabi--Yau denominators and AS-index rigidity}
\label{sec:cy-as-index}

\subsection{Polynomial denominator and local Smith data}

We keep the notation of the introduction: $A$ is a locally finite positively graded
elementary twisted Calabi--Yau algebra of dimension $d$ with $A_0=\bigoplus_{i=1}^r\K e_i$,
the matrix Hilbert series $C_A(x)$ is the inverse of the polynomial homological denominator
$q_A(x)$ \cite[Proposition~4.2(1)]{ReyesRogalski2019}, and $M_A=q_A(1+t)$.  Let $\mu$ be a
Nakayama automorphism and let $P$ be the permutation matrix of its action on the primitive
idempotents, with $P_{ij}=\delta_{\mu(i),j}$; write
$L=\operatorname{diag}(\ell_1,\ldots,\ell_r)$ for the AS-index matrix.  Reyes--Rogalski
prove that
\begin{equation}
\label{eq:rr-cy-functional}
  q_A(x)=(-1)^dPx^Lq_A(x^{-1})^T,
  \qquad
  q_A(x)Px^L=Px^Lq_A(x),
\end{equation}
see \cite[Proposition~6.4]{ReyesRogalski2019}.

Since $M_A\in M_r(R)\cap GL_r(F)$, ordinary Smith theory over the DVR
$R$ gives unique integers $0\leq a_1\leq\cdots\leq a_r$ and matrices $U,V\in GL_r(R)$ such
that
\begin{equation}
\label{eq:cy-local-smith-form}
  UM_AV=\operatorname{diag}(t^{a_1},\ldots,t^{a_r}),
\end{equation}
so that $\mathcal G_A=\coker(M_A)\cong\bigoplus_iR/(t^{a_i})$.  We call
$(a_1,\ldots,a_r)$ the local Smith profile of $A$ at $1$.
For $1\leq s\leq r$, let $\delta_s(q_A)$ be the least $t$-adic
valuation of a nonzero $s\times s$ minor of $M_A$, with $\delta_0=0$.
Then
\begin{equation}
\label{eq:cy-determinantal-orders}
  \delta_s(q_A)=a_1+\cdots+a_s.
\end{equation}
Set $\varepsilon_A(x)=\det q_A(x)$.  The Smith form gives
\[
  \length_R\mathcal G_A=\sum_i a_i
  =\ord_{x=1}\varepsilon_A(x).
\]

\begin{lemma}[Finite-growth Smith degree]
  \label{lem:cy-smith-growth}
  If $\GKdim A<\infty$, then
  \begin{equation}
    \label{eq:cy-gk-determinantal-order}
    \GKdim A= a_r =\delta_r(q_A)-\delta_{r-1}(q_A).
  \end{equation}
\end{lemma}

\begin{proof}
  Every entry of $C_A(x)=q_A(x)^{-1}$ is a rational series with nonnegative integral
  coefficients.  In the finite-growth case, Reyes--Rogalski
  \cite[Lemma~2.7(1)--(2)]{ReyesRogalski2019} identify the GK dimension of each nonzero
  entry with its pole order at $x=1$.  Since the Hilbert series of $A$ is the sum of the
  entries of $C_A$, $\GKdim A$ is the largest of these pole orders.  From
  \eqref{eq:cy-local-smith-form}, $M_A^{-1}=V\operatorname{diag}(t^{-a_i})U$; because $U$
  and $V$ are invertible over $R$, the largest pole order among its entries is $a_r$.  The
  final equality follows from \eqref{eq:cy-determinantal-orders}.
\end{proof}

\subsection{Functional equations and cycle averages}

The second identity in \eqref{eq:rr-cy-functional} already imposes a global restriction on
the AS-index vector.  For a cycle $C$ of the permutation $\mu$, define its average AS-index
by $\overline\ell_C=|C|^{-1}\sum_{i\in C}\ell_i$.  Partition the cycles by this value; for
each $c$, let $I_c$ be their union and put $e_c=\sum_{i\in I_c}e_i$.

\begin{theorem}[Nakayama-cycle average decomposition]
  \label{thm:cycle-average-decomposition}
  Every $e_c$ is central and
  \begin{equation}
    \label{eq:cycle-average-product}
    A\cong\prod_c e_cAe_c.
  \end{equation}
\end{theorem}

\begin{proof}
  Let $N$ be the order of $P$ and set $D(x)=Px^L$.  With our matrix convention,
  $Pe_j=e_{\mu^{-1}(j)}$.  Hence $D(x)^Ne_j=x^{\kappa_j}e_j$, where
  $\kappa_j=\sum_{q=0}^{N-1}\ell_{\mu^{-q}(j)}$.  If $j$ lies in a cycle $C$ of length $m$,
  each vertex of $C$ occurs $N/m$ times in this sum, so $\kappa_j=N\overline\ell_C$.  Thus
  $D(x)^N=x^K$ for a diagonal integral matrix
  $K=\operatorname{diag}(\kappa_1,\ldots,\kappa_r)$ whose entries are constant exactly on
  unions of cycles having the same average.

  Since $q_A$ commutes with $D(x)$, it commutes with $x^K$.  The $(i,j)$ entry of this
  identity is $(x^{\kappa_j}-x^{\kappa_i})q_A(x)_{ij}=0$.  Whenever $\kappa_i\neq\kappa_j$,
  the first factor is nonzero in $\K(x)$, hence $q_A(x)_{ij}=0$.  Therefore $q_A$ is block
  diagonal for the partition $\{I_c\}$, and so is its inverse $C_A$.

  By definition, $C_A(x)_{ij}=\Hilb_x(e_iAe_j)$.  A zero entry therefore means $e_iAe_j=0$.
  Thus $e_cAe_{c'}=0=e_{c'}Ae_c$ for $c\neq c'$, which makes each $e_c$ central and yields
  \eqref{eq:cycle-average-product}.
\end{proof}

Consequently, if $A$ is ring-indecomposable, all Nakayama cycles have the same average
AS-index, necessarily $\overline\ell_A=r^{-1}\sum_{i=1}^r\ell_i$.

\begin{remark}
  The arithmetic mean in Theorem~\ref{thm:cycle-average-decomposition} is the same averaging
  operation used in the average Gorenstein parameter of Iyama--Kimura--Ueyama
  \cite[Definition~4.3]{IyamaKimuraUeyama2026}, up to the usual left/right sign convention
  for grading shifts.  Their \cite[Theorem~4.7]{IyamaKimuraUeyama2026} concerns adjustment
  of individual parameters under graded Morita equivalence.  The theorem above instead
  follows directly from the denominator commutation relation and records how unequal cycle
  averages force an actual product decomposition.
\end{remark}

\subsection{Local normalization of the AS-index}

Assume for the rest of this subsection that $A$ is ring-indecomposable, and write
$\overline\ell_A$ for the common cycle average from
Theorem~\ref{thm:cycle-average-decomposition}.  Put $M=q_A(1+t)$ and $D=P(1+t)^L$.  Then
the two identities
\eqref{eq:rr-cy-functional} become
\begin{equation}
\label{eq:local-cy-functional-pre-normalization}
  M=(-1)^dDM^\dagger,
  \qquad MD=DM.
\end{equation}

\begin{theorem}[Local AS-index normalization]
  \label{thm:local-as-index-normalization}
  There exists a diagonal $G\in GL_r(R)$ with $G(0)=I_r$ and $G^\dagger=G^{-1}$ such that
  \begin{equation}
    \label{eq:local-as-index-normalization}
    G^{-1}DG=(1+t)^{\overline\ell_A}P.
  \end{equation}
\end{theorem}

\begin{proof}
  For each vertex choose a rational number $s_i$ satisfying
  $s_j+\ell_j-s_{\mu^{-1}(j)}=\overline\ell_A$.  Such a choice can be made cycle by cycle:
  choose one $s_i$ arbitrarily and solve successively around the cycle.  The final equation
  is consistent precisely because the sum of the $\ell_i$ on that cycle is its length times
  $\overline\ell_A$.

  Set $G=\operatorname{diag}((1+t)^{s_1},\ldots,(1+t)^{s_r})$.  Characteristic zero makes
  every rational binomial power a well-defined unit of $R$, and $G(0)=I_r$.  Since the
  involution sends $1+t$ to $(1+t)^{-1}$, one has $G^\dagger=G^{-1}$.  On the basis vector
  $e_j$, $G^{-1}DG$ acts by the scalar $(1+t)^{s_j+\ell_j-s_{\mu^{-1}(j)}}$ followed by $P$,
  proving \eqref{eq:local-as-index-normalization}.

\end{proof}

For later use, set $M^\sharp=G^{-1}MG$ and $D^\sharp=G^{-1}DG$.  Conjugating
\eqref{eq:local-cy-functional-pre-normalization} and using $G^\dagger=G^{-1}$ gives
$M^\sharp=(-1)^dD^\sharp(M^\sharp)^\dagger$ and $M^\sharp D^\sharp=D^\sharp M^\sharp$.
Since $G\in GL_r(R)$, the matrices $M$ and $M^\sharp$ have the same local Smith profile, and
$[v]\mapsto[G^{-1}v]$ identifies their cokernels.

The theorem explains why the individual entries of $L$ disappear from the exact-layer
restrictions below.  After a unitary local change of basis, the Calabi--Yau multiplier
depends on the AS-index only through the common average.  No integrality of
$\overline\ell_A$ is required locally.  If $N$ is the order of $P$, however,
$N\overline\ell_A$ is an integer because $N$ is divisible by the length of every Nakayama
cycle.

\section{Nakayama-equivariant local Smith theory}
\label{sec:twisted-cy}

\subsection{The local linking form}

Throughout this section, $A$ is a locally finite positively graded elementary twisted
Calabi--Yau algebra of dimension $d$ as in Section~\ref{sec:cy-as-index}.  We first assume
that $A$ is ring-indecomposable so that Theorem~\ref{thm:local-as-index-normalization}
applies; the parity statements extend componentwise to finite products.  We replace $(M,D)$
by the normalized pair of Theorem~\ref{thm:local-as-index-normalization} and drop the
superscript $\sharp$.  Thus $D=(1+t)^{\overline\ell_A}P$, $D^\dagger=D^{-1}$, and
\begin{equation}
\label{eq:local-cy-functional}
  M=(-1)^dDM^\dagger,
  \qquad MD=DM.
\end{equation}
The second identity lets $D$ descend to an $R$-linear automorphism
\[
\nu_A:\mathcal G_A\longrightarrow\mathcal G_A,
  \qquad
  \nu_A([v])=[Dv].
\]
For $v,w\in R^r$, set
\begin{equation}
\label{eq:twisted-linking-form}
  \langle[v],[w]\rangle_A
  :=v^\dagger M^{-1}w\pmod R.
\end{equation}

\begin{theorem}[Local linking form]
  \label{thm:twisted-cy-linking}
  The pairing \eqref{eq:twisted-linking-form} is well defined and nonsingular.
\end{theorem}

\begin{proof}
  From \eqref{eq:local-cy-functional},
  \begin{equation}
    \label{eq:dagger-M}
    M^\dagger=(-1)^dD^{-1}M.
  \end{equation}
  If $v$ is replaced by $v+Ma$, the change in \eqref{eq:twisted-linking-form} is
  $(-1)^da^\dagger D^{-1}w\in R$; if $w$ is replaced by $w+Mb$, the change is
  $v^\dagger b\in R$.  Thus the pairing is well defined.

  Let $\overline T$ denote the $\iota$-twist of an $R$-module $T$:
  $r\cdot_{\overline T}\xi:=\iota(r)\xi$.  The adjoint of \eqref{eq:twisted-linking-form} is
  the $R$-linear map
  \[
    \overline{\mathcal G_A}\longrightarrow
    \Hom_R(\mathcal G_A,F/R),
    \qquad
    \xi\longmapsto\langle\xi,-\rangle_A.
  \]
  If $\langle[v],[w]\rangle_A=0$ for every $w$, then $v^\dagger M^{-1}\in R^{1\times r}$.
  Taking $\dagger$ gives $(M^\dagger)^{-1}v\in R^r$, and \eqref{eq:dagger-M} identifies this
  vector, up to the sign $(-1)^d$, with $M^{-1}Dv$.  Since $D$ commutes with $M$ and is
  invertible over $R$, it follows that $v\in MR^r$.  The same coordinate argument in the
  second variable shows that the right radical is zero.  Hence the adjoint is injective.
  Both source and target have the same finite length: for every $a\geq1$,
  $\Hom_R(R/(t^a),F/R)\cong R/(t^a)$.  The adjoint is therefore an isomorphism, proving
  nonsingularity.
\end{proof}

The pairing is independent of the normalization chosen above.  Indeed, before dropping the
superscript $\sharp$, the cokernel isomorphism $\Phi([v])=[G^{-1}v]$ satisfies
$\langle\Phi([v]),\Phi([w])\rangle_{M^\sharp} =v^\dagger M^{-1}w\pmod R$, because
$(G^{-1})^\dagger=G$ and $(M^\sharp)^{-1}=G^{-1}M^{-1}G$.

The functional equation gives the twisted symmetry
\begin{equation}
\label{eq:twisted-linking-symmetry}
  \iota\!\left(\langle\eta,\xi\rangle_A\right)
  =(-1)^d\langle\xi,\nu_A(\eta)\rangle_A.
\end{equation}
Indeed,
\[
  \iota\!\left(w^\dagger M^{-1}v\right)
  =v^\dagger(M^\dagger)^{-1}w
  =(-1)^dv^\dagger M^{-1}Dw.
\]

\subsection{Exact Smith layers and Nakayama symmetry}

For $a\geq1$ define the auxiliary, or exact Smith, layer
\begin{equation}
\label{eq:exact-smith-layer}
  V_a(A):=\Delta_a(\mathcal G_A)
  :=\frac{\mathcal G_A[t^a]}
          {\mathcal G_A[t^{a-1}]+t\mathcal G_A[t^{a+1}]}.
\end{equation}
If $\mathcal G_A\cong\bigoplus_iR/(t^{a_i})$, then
\[
m_a(A):=\dim_\K V_a(A)=\#\{i:a_i=a\}.
\]

\begin{lemma}[Nakayama action on exact layers]
  The automorphism $\nu_A$ preserves $\mathcal G_A[t^m]$ for every $m$ and therefore induces
  an automorphism $\nu_a\in\operatorname{GL}(V_a(A))$ for every $a\geq1$.
\end{lemma}

\begin{proof}
  The map $\nu_A$ is $R$-linear, so it commutes with multiplication by $t$ and preserves
  every kernel $\ker(t^m)$.  It therefore preserves both summands in the denominator of
  \eqref{eq:exact-smith-layer} and induces the asserted automorphism.
\end{proof}

For $\xi,\eta\in\mathcal G_A[t^a]$, the linking value has pole order at most $a$.  Let
$[t^{-a}]$ denote the coefficient of $t^{-a}$ in $F/R$ and define
\[
\beta_a(\bar\xi,\bar\eta)
  :=[t^{-a}]\langle\xi,\eta\rangle_A.
\]
This is the classical DVR auxiliary-form construction; compare Levine
\cite[pp.~48--53]{Levine1980} and Orson \cite[Definition~3.6 and Theorem~3.8]{Orson2015}.

\begin{proposition}[Auxiliary form]
  The form $\beta_a$ is nondegenerate on $V_a(A)$.
\end{proposition}

\begin{proof}
  Sesquilinearity shows that the coefficient of $t^{-a}$ vanishes whenever either argument
  lies in $\mathcal G_A[t^{a-1}]$ or in $t\mathcal G_A[t^{a+1}]$, since multiplication by
  $t$, or by $\iota(t)$ in the first variable, lowers the pole order; hence $\beta_a$ is
  well defined.  Nondegeneracy follows from Theorem~\ref{thm:twisted-cy-linking} by the
  standard argument of \cite[Theorem~3.8]{Orson2015}: since $F/R$ is injective over $R$, a
  length count gives ${}^\perp\mathcal G_A[t^a]=t^a\mathcal G_A$, so a $\beta_a$-radical
  class $\xi$ satisfies $\iota(t)^{a-1}\xi\in t^a\mathcal G_A$; as $\iota(t)^{a-1}$ is a
  unit multiple of $t^{a-1}$, writing $t^{a-1}\xi=t^a\zeta$ places $\xi$ in
  $\mathcal G_A[t^{a-1}]+t\mathcal G_A[t^{a+1}]$.  The right radical is treated in the same
  way.
\end{proof}

Since $\iota(t^{-a})=(-1)^at^{-a}+$ terms of pole order $<a$, the leading coefficient of
\eqref{eq:twisted-linking-symmetry} gives
\begin{equation}
\label{eq:smith-layer-symmetry}
  \beta_a(y,x)=(-1)^{d+a}\beta_a(x,\nu_a y).
\end{equation}
The sign is the usual auxiliary-form shift produced by an involution with
$\iota(t)/t\equiv-1\pmod t$; compare \cite[Theorem~3.8]{Orson2015}.  What is specific here
is the additional operator $\nu_a$.  If $B_a$ is a Gram matrix for $\beta_a$ and $N_a$ is
the matrix of $\nu_a$, then \eqref{eq:smith-layer-symmetry} says $B_a^T=(-1)^{d+a}B_aN_a$.
Transposing this identity and substituting it back gives $B_a=N_a^TB_aN_a$.  Thus
\begin{equation}
\label{eq:smith-layer-isometry}
  \beta_a(\nu_a x,\nu_a y)=\beta_a(x,y).
\end{equation}

\begin{lemma}[Finite order on exact layers]
  \label{lem:nakayama-finite-order}
  Let $N$ be the order of the permutation matrix $P$.  Then $\nu_a^N=1$ on every $V_a(A)$.
\end{lemma}

\begin{proof}
  By local normalization, $D=(1+t)^{\overline\ell_A}P$.  Since the scalar factor is a
  central unit, $P=(1+t)^{-\overline\ell_A}D$ also descends to $\mathcal G_A$.  On
  $\mathcal G_A[t^a]$, multiplication by a unit congruent to $1$ modulo $t$ induces the
  identity on $V_a(A)$, because its difference from the identity has image in
  $t\mathcal G_A[t^a]\subseteq t\mathcal G_A[t^{a+1}]$.  Thus the operators induced by $D$
  and $P$ on $V_a(A)$ coincide.  Since $P^N=I_r$, the assertion follows.
\end{proof}

In characteristic zero, $z^N-1$ is separable.  After extending scalars to an algebraic
closure, $\nu_a$ is therefore semisimple with root-of-unity eigenvalues.  Write
$V_{a,\lambda}$ for the corresponding eigenspace.

\begin{corollary}[Reciprocal eigenspaces]
  \label{cor:reciprocal-eigenspaces}
  For every eigenvalue $\lambda$ of $\nu_a$, $\dim V_{a,\lambda}=\dim V_{a,\lambda^{-1}}$.
\end{corollary}

\begin{proof}
  Equation~\eqref{eq:smith-layer-isometry} shows that $V_{a,\lambda}$ pairs only with
  $V_{a,\lambda^{-1}}$; nondegeneracy makes the pairing perfect between them.
\end{proof}

\begin{corollary}[Self-reciprocal parity]
  If $\lambda\in\{1,-1\}$ and $(-1)^{d+a}\lambda=-1$, then $\dim V_{a,\lambda}$ is even.
\end{corollary}

\begin{proof}
  On the self-reciprocal eigenspace, \eqref{eq:smith-layer-symmetry} makes the nondegenerate
  restriction of $\beta_a$ alternating.
\end{proof}

Taking determinants in the Gram-matrix identity above gives
\begin{equation}
  \label{eq:smith-layer-determinant}
  \det(\nu_a)=(-1)^{(d+a)m_a(A)}.
\end{equation}

There is also a determinant-level parity statement.  Put
$\mathsf L_A=\length_R\mathcal G_A$.  Taking determinants in \eqref{eq:rr-cy-functional}
gives
\[
  \varepsilon_A(x)
  =(-1)^{dr}\det(P)x^{\operatorname{tr}L}\varepsilon_A(x^{-1}).
\]
Write $\varepsilon_A(1+t)=ct^{\mathsf L_A}+O(t^{\mathsf L_A+1})$ with $c\neq0$.  Since
$(1+t)^{-1}-1=-t+O(t^2)$ and $(1+t)^{\operatorname{tr}L}$ is a unit with constant term one,
comparison of leading terms gives
\begin{equation}
\label{eq:total-length-parity}
  (-1)^{\mathsf L_A}=(-1)^{dr}\det(P).
\end{equation}

After extending scalars to an algebraic closure, put $m_{a,-}(A)=\dim V_{a,-1}$, let
$m_0(A)=\#\{i:a_i=0\}$, and write $\det(P)=(-1)^{\epsilon(P)}$ with $\epsilon(P)\in\{0,1\}$.

\begin{theorem}[Equivariant parity of the full local Smith profile]
  \label{thm:equivariant-full-profile-parity}
  For every $a>0$,
  \begin{equation}
    \label{eq:minus-eigenspace-congruence}
    m_{a,-}(A)\equiv(d+a)m_a(A)\pmod2,
  \end{equation}
  and
  \begin{equation}
    \label{eq:full-profile-congruence}
    d\,m_0(A)+\sum_{a>0}m_{a,-}(A)
    \equiv\epsilon(P)\pmod2.
  \end{equation}
\end{theorem}

\begin{proof}
  Equation~\eqref{eq:smith-layer-determinant} gives $\det(\nu_a)=(-1)^{(d+a)m_a(A)}$.  By
  Corollary~\ref{cor:reciprocal-eigenspaces}, eigenvalues $\lambda\notin\{1,-1\}$ occur in
  reciprocal pairs with equal multiplicity, so their contributions to the determinant
  cancel.  Thus $\det(\nu_a)=(-1)^{m_{a,-}(A)}$, proving
  \eqref{eq:minus-eigenspace-congruence}.

Now sum \eqref{eq:minus-eigenspace-congruence} over $a>0$.  Since
\[
  \sum_{a>0}m_a(A)=r-m_0(A),
  \qquad
  \mathsf L_A=\sum_{a>0}a\,m_a(A),
\]
we obtain
\[
  \sum_{a>0}m_{a,-}(A)
  \equiv d(r-m_0(A))+\mathsf L_A\pmod2.
\]
Equation~\eqref{eq:total-length-parity} says $\mathsf L_A\equiv dr+\epsilon(P)\pmod2$, and
substitution gives \eqref{eq:full-profile-congruence}.
\end{proof}

The congruences in Theorem~\ref{thm:equivariant-full-profile-parity} are additive over ring
factors.  Indeed, the homological denominator, exact layers, Nakayama operators, and the
quantities $m_a$, $m_{a,-}$, $m_0$, and $\epsilon(P)$ decompose componentwise.  Thus the
theorem applies to a general elementary twisted Calabi--Yau algebra by the product
decomposition of Theorem~\ref{thm:cycle-average-decomposition}.

\begin{corollary}[Odd-order Nakayama profile parity]
  \label{cor:odd-order-profile-parity}
  Assume that $P$ has odd order.  Then $m_a(A)$ is even for every $a\geq0$ with $d+a$ odd.
\end{corollary}

\begin{proof}
  Let $N$ be the order of $P$.  By Lemma~\ref{lem:nakayama-finite-order}, $\nu_a^N=1$ on
  every positive exact layer.  Since $N$ is odd, $-1$ is not an eigenvalue, so
  $m_{a,-}(A)=0$ for $a>0$.  Equation \eqref{eq:minus-eigenspace-congruence} gives the
  assertion for positive $a$.  An odd-order permutation is even, so $\det P=1$ and
  $\epsilon(P)=0$.  Equation~\eqref{eq:full-profile-congruence} therefore gives
  $d\,m_0(A)\equiv0\pmod2$, which is exactly the assertion at $a=0$.
\end{proof}

\begin{remark}[Comparison with palindromic Smith theory]
  If $P=I_r$ and $L=\ell I_r$, then $q_A(x)=(-1)^d x^\ell q_A(x^{-1})^T$ is an ordinary
  $T$-palindromic matrix polynomial of type $(-1)^d$.  Mackey--Mackey--Mehl--Mehrmann
  \cite[Theorem~7.6]{MackeyMackeyMehlMehrmann2011} show that at $x=1$ the odd partial
  multiplicities pair when the type is $+1$, while the even partial multiplicities,
  including zero, pair when the type is $-1$.  Thus
  Corollary~\ref{cor:odd-order-profile-parity} recovers their parity restriction in this
  scalar untwisted specialization.  The Nakayama-equivariant refinement is the additional
  content here; Theorem~\ref{thm:local-as-index-normalization} explains why the non-scalar
  part of the AS-index is locally removable before the exact layers are formed.
\end{remark}

At the maximal layer, if $\GKdim A<\infty$ and $g=\GKdim A=a_r>0$ by
Lemma~\ref{lem:cy-smith-growth}, Corollary~\ref{cor:odd-order-profile-parity} says that
$m_g(A)$ is even whenever $P$ has odd order and $d+g$ is odd.  In rank one $m_g(A)=1$, so
$g\equiv d\pmod2$; in the degree-one Artin--Schelter regular setting
\cite{ArtinSchelter1987} this agrees with Kabbaj \cite[Theorem~2.4]{Kabbaj2026}.  For
nontrivial $P$, Theorem~\ref{thm:equivariant-full-profile-parity} retains the finer Nakayama
eigenspace decomposition on this layer.

\section{Examples and applications}
\label{sec:examples}

We begin with a family in which the complete Nakayama-equivariant Smith data can be computed
explicitly, then turn to the small-rank classification applications.  The final subsection
gives an explicit realization of a four-vertex type that survives the local obstructions.

\subsection{Cyclic skew-group algebras in arbitrary dimension and rank}

Assume throughout this subsection that $\K$ is algebraically closed of characteristic zero
and contains a primitive $n$th root of unity $\zeta$.  Let
\[
  S=\K[X_1,\ldots,X_d],
  \qquad
  G=C_n=\langle g\rangle,
\]
and let $G$ act diagonally by
\[
g(X_i)=\zeta^{w_i}X_i,
  \qquad w_i\in\mathbb Z/n.
\]
Put $A=S\#\K G$ and
$s=w_1+\cdots+w_d\in\mathbb Z/n$.
For $a\in\mathbb Z/n$, let $P_a$ be the permutation matrix acting on
column coordinates by $e_j\mapsto e_{j+a}$.

The smash-product theorem of Reyes--Rogalski--Zhang
\cite[Theorem~4.1(b)--(c)]{ReyesRogalskiZhang2014} shows that $A$ is an elementary twisted
Calabi--Yau algebra of dimension $d$ with $n$ vertices.  For $j\in\mathbb Z/n$ put
\[
\alpha_j:=\#\{i:jw_i=0\text{ in }\mathbb Z/n\},
\]
and set
$v_j=\sum_{k\in\mathbb Z/n}\zeta^{-jk}e_k$.

\begin{proposition}[Equivariant Smith profile of a cyclic skew-group algebra]
  \label{prop:cyclic-skew-profile}
  The Nakayama and AS-index matrices are $P_s$ and $dI_n$, and
  \[
    q_A(x)=\prod_{i=1}^d(I_n-xP_{w_i}).
  \]
  The local Smith exponents are $\{\alpha_j:j\in\mathbb Z/n\}$, and
  \[
    V_a(A)=\bigoplus_{\alpha_j=a}\K\bar v_j,
    \qquad
    \nu_a(\bar v_j)=\zeta^{js}\bar v_j.
  \]
\end{proposition}

\begin{proof}
  The Nakayama automorphism is determined by the homological determinant in
  \cite[Theorem~4.1(b)--(c)]{ReyesRogalskiZhang2014}.  Since
  $\operatorname{Ext}^d_S(\K,\K)\cong\bigwedge^d(S_1^*)$, the action of $g$ on this
  one-dimensional space is multiplication by $\zeta^{-s}$.  With the antipode convention in
  \cite[Definition~3.7 and Remark~3.8]{ReyesRogalskiZhang2014}, this gives
  $\operatorname{hdet}(g)=\zeta^s$.  For the character idempotents
\[
  e_j=\frac1n\sum_{m=0}^{n-1}\zeta^{-jm}g^m,
\]
the winding part sends $e_j$ to $e_{j-s}$; in our column convention the Nakayama permutation
matrix is therefore $P_s$.

The Koszul resolution of $\K$ over $S$ is $G$-equivariant.  After decomposition by the
character idempotents, homological degree $p$ contributes
\[
  \sum_{\substack{I\subseteq\{1,\ldots,d\}\\|I|=p}}
  P_{\sum_{i\in I}w_i}
\]
with internal shift $p$.  Taking the alternating Euler sum gives
\[
  \sum_{p=0}^d(-x)^p
  \sum_{|I|=p}P_{\sum_{i\in I}w_i}
  =\prod_{i=1}^d(I_n-xP_{w_i}),
\]
and the top Koszul term gives the AS-index matrix $dI_n$.

The vectors $v_j$ form a Fourier basis with $P_av_j=\zeta^{ja}v_j$.  Hence $q_A(x)$ acts on
the $j$th Fourier line by $\prod_i(1-\zeta^{jw_i}x)$, whose order of vanishing at $x=1$ is
$\alpha_j$.  Since the Fourier change of basis is constant and invertible, these orders are
the Smith exponents.  The corresponding local cokernel summand is $R/(t^{\alpha_j})$ up to a
unit, so $V_a(A)$ consists precisely of the lines with $\alpha_j=a$.  Finally,
$D=P_s(1+t)^d$, and $(1+t)^d$ induces the identity on every exact layer.  Thus
$\nu_a(\bar v_j)=\zeta^{js}\bar v_j$.
\end{proof}

Changing the antipode convention replaces $s$ by $-s$ and hence inverts every exact-layer
eigenvalue.  Since $\alpha_j=\alpha_{-j}$, the layer spectra, and therefore all parity data,
are unchanged.

As graded vector spaces $A\cong S\otimes_\K\K G$, so $\GKdim A=d$.  Correspondingly
$\alpha_0=d$, and Lemma~\ref{lem:cy-smith-growth} identifies this with the largest local
exponent.  Proposition~\ref{prop:cyclic-skew-profile} records substantially more than this
extremal value.

\begin{example}[An odd-order Nakayama action]
  Take $d=3$, $n=6$, and $(w_1,w_2,w_3)=(0,1,3)$.  Then $s=4$, so $P=P_4$ has order three.
  The Fourier exponents are $(3,1,2,1,2,1)$, giving the ordered Smith profile
  $(1,1,1,2,2,3)$.  If $\omega=\zeta^2$, then
  \[
    \operatorname{Spec}(\nu_1)=\{1,\omega,\omega^{-1}\},\qquad
    \operatorname{Spec}(\nu_2)=\{\omega,\omega^{-1}\},\qquad
    \operatorname{Spec}(\nu_3)=\{1\}.
  \]
  In particular, the exponent $2$, whose parity is opposite to $d=3$, occurs twice, as
  required by Corollary~\ref{cor:odd-order-profile-parity}.
\end{example}

\begin{example}[An even-order Nakayama action]
  Take $d=4$, $n=6$, and $(w_1,w_2,w_3,w_4)=(0,1,3,3)$.  Here $s=1$, so $P$ is a $6$-cycle.
  The Fourier exponents are $(4,1,3,1,3,1)$ and the ordered profile is $(1,1,1,3,3,4)$.  The
  exact-layer spectra are
  \[
    \operatorname{Spec}(\nu_1)=\{\zeta,-1,\zeta^{-1}\},\qquad
    \operatorname{Spec}(\nu_3)=\{\zeta^2,\zeta^{-2}\},\qquad
    \operatorname{Spec}(\nu_4)=\{1\}.
  \]
  Thus $m_{1,-}=1$, while $m_{3,-}=m_{4,-}=0$.  The congruence
  \eqref{eq:minus-eigenspace-congruence} therefore uses information that is not contained in
  the multiplicities $m_a$ alone.
\end{example}

\subsection{Three-vertex profile filters}
\label{ex:profile-obstruction}

Suppose that $d=3$, $r=3$, $A$ is ring-indecomposable, the Nakayama permutation has odd
order, and $\GKdim A=3$.  Lemma~\ref{lem:cy-smith-growth} gives $a_3=3$.  Since $P$ is
either the identity or a $3$-cycle, Corollary~\ref{cor:odd-order-profile-parity} requires
every even exponent, including zero, to occur with even multiplicity.  The possible profiles
are therefore
\[
(0,0,3),\quad
  (1,1,3),\quad
  (1,3,3),\quad
  (2,2,3),\quad
  (3,3,3).
\]

This restriction is strictly finer than maximal pole order and determinant parity.  For
example, the hypothetical profile $(0,2,3)$ has maximal exponent three and, for a $3$-cycle,
\[
  (-1)^5=(-1)^{3\cdot3}\det(P),
\]
so \eqref{eq:total-length-parity} holds; nevertheless zero and two each occur with odd
multiplicity, contradicting Corollary~\ref{cor:odd-order-profile-parity}.

Several surviving profiles occur in the cyclic skew-group family of
Proposition~\ref{prop:cyclic-skew-profile}.  For $n=3$, the weights $(1,1,2)$, $(0,1,1)$,
and $(0,0,1)$ give respectively $(0,0,3)$, $(1,1,3)$, and $(2,2,3)$.  Both the identity and
a $3$-cycle can occur as Nakayama permutation for the first two profiles.  The profile
$(1,3,3)$ does not occur in this family, since the two nontrivial Fourier exponents
coincide.  The constant profile $(3,3,3)$ occurs for the trivial action, but then the
algebra is decomposable.  The latter phenomenon is intrinsic in the degree-one setting.

\begin{proposition}[Degree-one rigidity of the minimal exponent]
  \label{prop:minimal-exponent-bound}
  Let $A$ be a ring-indecomposable elementary twisted Calabi--Yau algebra of dimension $3$
  on $r\geq2$ vertices which is generated in degree one.  Then its smallest local Smith
  exponent satisfies $a_1\leq2$.
\end{proposition}

\begin{proof}
  By \cite[Proposition~5.2(5)]{ReyesRogalski2019}, the AS-index is scalar; write it as
  $sI_r$, and let $M_{ij}=\dim_\K e_iA_1e_j$.  The dimension $3$ denominator formula
  \cite[Proposition~8.2]{ReyesRogalski2019} becomes
  \begin{equation}
    \label{eq:degree-one-denominator-rigidity}
    q_A(x)
    =I_r-Mx+PM^Tx^{s-1}-Px^s ,
  \end{equation}
  with $s\geq3$.
  
  If $a_1\geq3$, then every entry of $q_A$ vanishes to order at least three at $x=1$, so
  $q_A''(1)=0$.  Differentiating \eqref{eq:degree-one-denominator-rigidity} twice gives
  \[
    (s-2)M^T=sI_r.
  \]
  Hence $M$ is diagonal.  Since $A$ is generated by $A_0$ and $A_1$, this forces $e_iAe_j=0$
  for $i\neq j$, contradicting ring-indecomposability.
\end{proof}

Thus $(3,\ldots,3)$ cannot occur for a ring-indecomposable degree-one algebra on more than
one vertex, without any growth hypothesis or bound on the superpotential degree.  The
restriction $r\geq2$ is necessary: $\K[X_1,X_2,X_3]$ has profile $(3)$.

We next compare the parity sieve with the three-vertex classification of Gaddis--Rogalski
\cite[Theorems~1.6 and~1.8]{GaddisRogalski2021}, denoting their superpotential-degree
parameter by $s$ and reserving $d=3$ for the Calabi--Yau dimension.  Under
\cite[Hypothesis~1.4]{GaddisRogalski2021}, one assumes GK dimension $3$ and superpotential
degree $s\in\{3,4\}$; by \cite[Theorem~1.2]{GaddisRogalski2021}, connected degree-one
twisted Calabi--Yau algebras of dimension $3$ arise from the corresponding twisted
superpotentials.  For odd-order Nakayama permutation, the relevant finite types are the
eight $3$-cycle cases of \cite[Proposition~5.1 and Theorem~1.6]{GaddisRogalski2021} and the
nine finite $P=I_3$ cases of \cite[Proposition~6.12 and
Theorem~1.8]{GaddisRogalski2021}.  Substitution into
\eqref{eq:degree-one-denominator-rigidity} and local Smith reduction at $x=1$ gives only
\begin{equation}
  \label{eq:three-vertex-classified-profiles}
  (0,0,3),\qquad
  (1,1,3),\qquad
  (2,2,3);
\end{equation}
the Markov family in \cite[Theorem~1.8]{GaddisRogalski2021}, realized by
\cite[Proposition~2.4]{GaddisRogalski2021}, has $q_A(1)=M^T-M\neq0$, hence $a_1=0$ and again
the profile $(0,0,3)$.  Under Hypothesis~1.4 the classification therefore removes the two
remaining parity survivors: $(3,3,3)$ is excluded intrinsically by
Proposition~\ref{prop:minimal-exponent-bound}, while the status of $(1,3,3)$ outside
Hypothesis~1.4 remains open.  A degree-one example would require superpotential degree at
least five; Gaddis--Rogalski \cite[paragraph following Hypothesis~1.4]{GaddisRogalski2021}
expect the restriction $s\in\{3,4\}$ to be extraneous, and weighted examples are not covered
by their classification.

Finally, the Smith language shortens one of the growth exclusions in that program.  For the
candidate of \cite[Lemma~6.10]{GaddisRogalski2021},
\[
  M=
  \begin{pmatrix}
    2&3&0\\
    1&0&2\\
    4&0&0
  \end{pmatrix},
  \qquad
  p(x)=I_3-Mx+M^Tx^2-x^3I_3,
\]
one has
\[
  \det p(x)=-(x^2+x+1)(x+1)^2(x-1)^5,
  \qquad
  \operatorname{rank}p(1)=2.
\]
Thus $\delta_1(p)=\delta_2(p)=0$ and $\delta_3(p)=5$, so the local profile is $(0,0,5)$.
Equation~\eqref{eq:cy-gk-determinantal-order} would then force GK dimension $5$ in the
finite-growth branch, contradicting \cite[Hypothesis~1.4]{GaddisRogalski2021}.  The original
proof obtains the same cancellation information from the adjugate of $p$; the local Smith
profile packages it directly.

\subsection{Realization of a starred four-vertex candidate}
\label{ex:four-vertex-realization}

We now realize the first starred case of superpotential degree four in \cite[Theorem~4.10
and the paragraph following it]{GaddisLamkinNguyenWright2024}.  It has Nakayama permutation
$P=(12)(34)$ and incidence matrix
\[
  M=
  \begin{pmatrix}
    1&0&1&0\\
    0&1&0&1\\
    0&1&0&1\\
    1&0&1&0
  \end{pmatrix}.
\]
Gaddis--Lamkin--Nguyen--Wright \cite[Theorem~4.10 and the paragraph following
it]{GaddisLamkinNguyenWright2024} leave this case unresolved and state that they expect it
not to occur.  The construction below shows that it does.  The prospective homological
denominator $p(x)=I_4-Mx+PM^Tx^3-Px^4$ has local profile $(0,1,2,3)$; its $P=+1$ sector has
exponents $(1,3)$ and its $P=-1$ sector has exponents $(0,2)$, so the equivariant parity
restrictions are compatible with the candidate.

Consider
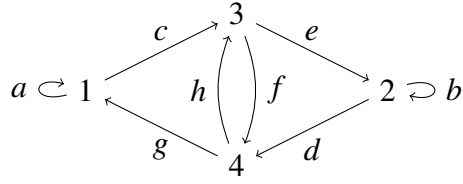
\[\begin{tikzpicture}
    \node(1) at (0,1) {$1$};
    \node(2) at (4,1) {$2$};
    \node(3) at (2,2) {$3$};
    \node(4) at (2,0) {$4$};

    \path[->]
    (1) edge[loop left] node {$a$} ()
    (2) edge[loop right] node {$b$} ()
    (1) edge node[above] {$c$} (3)
    (2) edge node[below] {$d$} (4)
    (3) edge node[above] {$e$} (2)
    (3) edge[bend left=20] node[right] {$f$} (4)
    (4) edge node[below] {$g$} (1)
    (4) edge[bend left=20] node[left] {$h$} (3);
\end{tikzpicture}\]
and set
\begin{align*}
  W_0={}&aace+aceb+cebb+ebbd+bbdg+bdga+dgaa+gaac\\
        &+cfhe+fhed+hedh+edhf+dhfg+hfgc+fgcf+gcfh.
\end{align*}
Its eight left arrow derivatives are linearly independent.  Put
\[
  R=\operatorname{span}_\K\{\partial_\gamma W_0:\gamma\in Q_1\},
  \qquad
  A_{W_0}=T_E(V)/(R).
\]

\begin{lemma}[Degree-four overlap]
  \label{lem:starred-degree-four-overlap}
  Let $W_4=\operatorname{span}_\K\{e_iW_0:1\leq i\leq4\}$.  Then
  \[
    W_4=(R\otimes_EV)\cap(V\otimes_ER)
    \subset V^{\otimes_E4}.
  \]
\end{lemma}

\begin{proof}
  Writing $r_\gamma=\partial_\gamma W_0$, direct expansion gives
  \[
    W_0=\sum_{\gamma\in Q_1}\gamma r_\gamma
    =\sum_{\gamma\in Q_1}r_\gamma\tau(\gamma),
    \qquad
    \tau=(a\,b)(c\,d)(e\,g)(f\,h),
  \]
  so $W_4$ lies in the intersection.  The two tensor products have common $E$-bimodule
  support only in the four components $e_i(-)e_{P(i)}$.  In each component their natural
  $2$-dimensional spans meet in the line $\K e_iW_0$, as follows immediately by comparing
  the four paths in the two bases.  Summing over the vertices proves the claim.
\end{proof}

We identify $A_{W_0}$ with an iterated skew-group algebra.  Let
\[
  \mathsf D=D(0,-1)
  =\K\langle x,y\rangle/(x^2y+yx^2,\;xy^2+y^2x).
\]
By \cite[Lemma~1.1]{ChenKirkmanZhang2017}, $\mathsf D$ is a connected graded noetherian
Artin--Schelter regular domain \cite{ArtinSchelter1987} of global dimension $3$ and GK
dimension $3$, hence skew Calabi--Yau of dimension $3$ by
\cite[Lemma~1.2]{ReyesRogalskiZhang2014}.

Let the first copy of $C_2$ act by $x\mapsto x$ and $y\mapsto-y$, and put
$\mathsf B=\mathsf D\#\K C_2$.  For the primitive idempotents $e_\pm$ of $\K C_2$, write
$X_\pm=e_\pm x=xe_\pm$ and $Y_\pm=e_\pm y=ye_\mp$.  Then $X_\pm$ are loops and
$Y_\pm:e_\pm\to e_\mp$, with relations
\begin{equation}
\label{eq:B-presentation}
  X_\pm^2Y_\pm+Y_\pm X_\mp^2,
  \qquad
  X_\pm Y_\pm Y_\mp+Y_\pm Y_\mp X_\pm .
\end{equation}
There is a second graded involution of $\mathsf B$, fixing the idempotents, which fixes
$X_+,Y_+$ and negates $X_-,Y_-$.  Put $\mathsf A=\mathsf B\#\K C_2$, and let $f_\pm$ be the
primitive idempotents of the second group algebra.

\begin{lemma}[Iterated smash identification]
  \label{lem:iterated-smash-identification}
  There is a graded algebra isomorphism $A_{W_0}\cong\mathsf A$ given on
  vertices by
  \[
    e_1\mapsto e_+f_+,
    \quad e_2\mapsto e_+f_-,
    \quad e_3\mapsto e_-f_+,
    \quad e_4\mapsto e_-f_-,
  \]
  and on arrows by
  \[
    \begin{alignedat}{4}
      a&\mapsto X_+f_+,&\quad b&\mapsto X_+f_-,&\quad
      c&\mapsto Y_+f_+,&\quad d&\mapsto Y_+f_-,\\
      e&\mapsto Y_-f_-,&
      f&\mapsto X_-f_-,&
      g&\mapsto Y_-f_+,&
      h&\mapsto X_-f_+ .
    \end{alignedat}
  \]
\end{lemma}

\begin{proof}
  If $\alpha$ is an eigenarrow of the second involution with eigenvalue
  $\lambda\in\{\pm1\}$, then $\alpha f_\psi$ has source $(s(\alpha),\lambda\psi)$ and target
  $(t(\alpha),\psi)$.  The displayed assignment therefore reproduces the quiver with
  incidence matrix $M$.  The four relations in \eqref{eq:B-presentation} split into eight
  $f_\pm$-components, which become exactly the eight derivatives $\partial_\gamma W_0$.  For
  example, $(X_+^2Y_++Y_+X_-^2)f_+=aac+cfh=\partial_gW_0$.  Thus the two path-algebra
  presentations have the same relation ideal.
\end{proof}

\begin{theorem}[Realization of the starred candidate]
  \label{thm:starred-realization}
  The algebra $A_{W_0}$ is twisted Calabi--Yau of dimension $3$, has $\GKdim A_{W_0}=3$, and
  has type $(M,(12)(34),4)$.  In particular, the first starred case of
  \cite[Theorem~4.10]{GaddisLamkinNguyenWright2024} is realizable.
\end{theorem}

\begin{proof}
  The group algebra $\K C_2$ is Calabi--Yau of dimension zero.  Since $\mathsf D$ is skew
  Calabi--Yau of dimension $3$, \cite[Proposition~7.3.1]{LeMeur2019} applied to the two
  actions shows successively that $\mathsf B$ and $\mathsf A$ are skew Calabi--Yau of
  dimension $3$; being graded, they are then graded twisted Calabi--Yau of dimension $3$
  \cite[Theorem~4.2]{ReyesRogalski2022}.  Lemma~\ref{lem:iterated-smash-identification}
  therefore makes $A_{W_0}$ twisted Calabi--Yau of dimension $3$.

  The quiver is connected and the minimal relations are cubic, so
  \cite[Theorem~1.2]{GaddisRogalski2021} makes $A_{W_0}$ $3$-Koszul.  Its degree-four Koszul
  syzygy is $(R\otimes_EV)\cap(V\otimes_ER)=W_4$ by
  Lemma~\ref{lem:starred-degree-four-overlap}.  The four summands $e_iW_0$ lie in the
  bimodule components $e_i(-)e_{P(i)}$; hence the top term of the minimal vertex-simple
  resolutions has Nakayama permutation $P=(12)(34)$ and internal degree four.  Thus the type
  is $(M,P,4)$.  Finally, $\mathsf A\cong\mathsf D\otimes\K C_2\otimes\K C_2$ as graded
  vector spaces, so $\GKdim\mathsf A=\GKdim\mathsf D=3$.
\end{proof}

By \eqref{eq:degree-one-denominator-rigidity}, the realized algebra has
\[
  \mathcal E_{A_{W_0}}(x)=I_4-Mx+PM^Tx^3-Px^4.
\]
Its local profile is therefore $(0,1,2,3)$, with exact-layer actions $\nu_1=1$, $\nu_2=-1$,
and $\nu_3=1$, exactly as predicted by Theorem~\ref{thm:equivariant-full-profile-parity}.

\subsection*{Use of large language models}

We used Large Language Models (LLMs) in three supporting roles in the preparation of this
paper.
\begin{enumerate}[(a)]

\item We used LLMs for copy editing. This includes suggestions concerning wording,
  organization, and presentation.  LLMs were never used as a source for verification of the
  arguments.
  
\item We used LLMs to help identify potentially relevant references.  We carefully sifted
  the suggestions and then checked every reference cited in the paper by hand.

\item We used LLMs for a directed and curated search for examples constrained by the
  theoretical results developed in the paper.  In particular, the three- and four-vertex
  examples in Section~\ref{sec:examples} were found through this process.  The resulting
  candidates were subsequently checked by using a sagemath program.

\end{enumerate}

The author takes full responsibility for the mathematical content of the
paper, including the correctness of all statements, arguments, computations,
and references.

\bibliographystyle{siam}
\bibliography{references}

\end{document}